\documentclass[11pt,a4paper,reqno]{amsart}
\usepackage{amsfonts}
\usepackage{graphicx}   
\usepackage{graphics}
\usepackage{epsfig}

\usepackage{anysize}
\marginsize{3cm}{3cm}{3cm}{3cm}

\newtheorem{theorem}{Theorem}[section]

\newtheorem{lemma}{Lemma}[section]
\newtheorem{remark}{Remark}[section]

\newtheorem{definition}{Definition}[section]
\newtheorem{example}{Example}[section]

\makeatletter

\begin{document}

\setcounter{page}{1}
\begin{flushleft}
{\scriptsize }
\end{flushleft}
\bigskip
\bigskip

\title[\centerline{On $\Phi$-Schauder Frames
\hspace{0.5cm}}] {On $\Phi$-Schauder Frames}

\author[\hspace{0.7cm}\centerline{TWMS J. App. Eng.
Math. V..., N..., 2012}] {L.K.Vashisht}

\thanks {\noindent  Department of Mathematics, University of Delhi, Delhi, India.\\
\indent \,\,\, e-mail: lalitkvashisht@gmail.com\\
\indent \S \, Manuscript received Day Month 2011.\\
\indent \,\,\, TWMS Journal of Applied and Engineering Mathematics
Vol.1 No.2 \copyright\ I\c s\i k University, Department \\
\indent \,\,\, of Mathematics 2011; all rights reserved.}


\bigskip
\begin{abstract}
 In this short  note we introduce and study a particular type of Schauder frames, namely,  $\Phi$-Schauder frames.
\bigskip


\end{abstract}

\maketitle
\bigskip


\section{Introduction  and Preliminaries}

Frames for Hilbert spaces were introduced by Duffin and Schaeffer
\cite{6} in 1952, while addressing some deep problems in
non-harmonic Fourier series. Later, in 1986, Daubechies, Grossmann
and Meyer \cite{5} found new applications to wavelets and Gabor
transforms in which frames played an important role.\

Today, frames play important roles in many applications in
mathematics, science and engineering . In particular frames are
widely used in sampling theory, wavelet theory, wireless
communication, signal processing, image processing , differential
equations, filter banks, geophysics, quantum computing, wireless
sensor network, multiple-antenna code design and many more.  Reason
is that frames provides both great  liberties in the design of
vector space decompositions, as well as quantitative measure on the
computability and robustness of the corresponding reconstructions.
In the theoretical direction, powerful tools from operator theory
and Banach spaces are being employed to study frames. For a nice and
comprehensive survey on various types of frames,  one may refer to
\cite{1,4,10} and the references therein.

Coifman and Weiss \cite{3} introduced the notion of atomic decomposition for function spaces. Later ,
 Feichtinger and Gr\"ochenig \cite{7} extended this idea to Banach spaces. This concept was further
 generalized by Gr\"ochenig \cite{8} who introduced the notion of Banach frames for Banach spaces.
 Casazza, Han and Larson \cite{2} also carried out a study of atomic decompositions and Banach frames.
 Recently, various generalization of frames in Banach spaces have been introduced and studied.
 Han and Larson \cite{9} introduced Schauder frames for Banach spaces.
 Schauder frames were further studied in \cite{13,14}. The notion of retro Banach frames in  Banach spaces
 introduced and studied in \cite{12}. \\

Throughout this note  $E$ will denote an infinite dimensional Banach
space and  $E^*$ the conjugate space of $E$.

\begin{definition}[\cite{9}]
A pair $(x_n,f_n)$  $(\{x_n\}\subset E$,  $\{f_n\}\subset E^*)$ is
called a $\emph{Schauder frame}$ for $E$ if
\begin{align}
x=\sum_{n=1}^\infty f_n(x)x_n,\quad\text{for all} \ x\in E,
\end{align}
where the series converges in the norm topology of $E$.
\end{definition}

\begin{definition}[\cite{12}]
Let $E$ be a Banach space and  let $E_d$  be an associated Banach
space of scalar-valued sequences indexed by $\mathbb{N}$. Let
$\{x_n\}\subset E$ and $T:(E^*)_d\to E^*$ be given. The pair
$(\{x_n\},T)$ is called a \emph{retro Banach frame} for $E^*$ with
respect to $(E^*)_d$ if :
\begin{enumerate}
\item[{(i)}] $\{f(x_n)\}\in (E^*)_d$, for all $f\in E^*$.
\item[{(ii)}] There exist positive constants $A$ and $B$ with $0<A\le B<\infty$ such that
\begin{align}
A\|f\|_{E^*}\le \|\{f(x_n)\}\|_{(E^*)_d}\le
B\|f\|_{E^*},\quad\text{for all } \ f\in E^*.
\end{align}
\item[{(iii)}] $T$ is a bounded linear operator such that $T(\{f(x_n)\})=f$, for all $f\in E^*$.
\end{enumerate}
The positive constants $A$ and $B$ are called  the \emph{lower} and
\emph{upper frame bounds}, of the retro Banach frame $(\{x_n\},T)$,
respectively. The operator $T:(E^*)_d\to E^*$ is called the
\emph{reconstruction operator} (or the \emph{pre-frame operator})
and the inequality (2) is called the \emph{retro frame inequality}.
\end{definition}

\begin{lemma} Let $E$ be a Banach space and $\{f_n\}\subset E^*$ be a
 sequence such that \break $\{x\in E:f_n(x)=0$, for \ all \ n $\in
 \mathbb{N}\}$=$\{0\}$. Then $E$ is linearly isometric to the \break Banach space
 $X=\{\{f_n(x)\}:x\in E\}$, where the norm is given by
$\|\{f_n(x)\}\|_X= \|{x}\|_E$, $x\in E$.
\end{lemma}

In this note  we introduce and study a particular type of Schauder
frames, namely,  $\Phi$-Schauder frames.   Necessary and / or
sufficient conditions for a  Schauder frame  to be   $\Phi$-Schauder
frame have  been obtained.

\section{ $\Phi$-Schauder frames}

\begin{definition}A Schauder frame $(x_n,f_n)$ $\{x_n\}\subset E, \{f_n\}\subset E^*)$
for a Banach space $E$
 is said to be   \emph{  $\Phi$-Schauder frame }
 if $\inf\limits_{1\le n<\infty}\|f_n\|>0$ and  there exists a functional
$\Phi \in E^*$ such that $\Phi(x_n)=1$, for all $n \in \mathbb{N}.$

\end{definition}

The functional $\Phi$ is called the \emph{associated functional} of
the Schauder frame $(x_n,f_n)$.

To show existence of $\Phi$- Schauder frames , we have following
example.

\begin{example}
Let $E=l^1$ and  $\{e_n\}\subset E$ be the sequence of canonical
unit vectors.
\begin{enumerate}
\item[(a)] Let $\{x_n\}\subset E$, $\{f_n\}\subset E^*$ be sequences defined by
\begin{align*}
\left.\begin{array}{ll}
x_1=\frac{1}{2}e_1,\qquad\quad x_2= \frac{1}{2}e_1,&x_n=e_{n-1},\\[4pt]
f_1(x)=\xi_{1},\quad f_2(x)=\xi_{1},&f_n(x)=\xi_{n-1},\quad
x=\{\xi_n\}\in E
\end{array}\!\!\right\}, \ n=3,4,5,....
\end{align*}
Then $(x_n,f_n)$, is a Schauder frame for $E$ with
$\inf\limits_{1\le n<\infty}\|f_n\|>0$.\ Now\break
\mbox{$\Phi=(2,1,1,1,1,\ldots)\in E^*$} is such that $\Phi(x_n)=1$,
for all $n\in \mathbb{N}$. Hence  $(x_n,f_n)$ is a $\Phi$-Schauder
frame of
 for $E$.

\item[{(b)}]
Let $\{x_n\}\subset E$, $\{f_n\}\subset E^*$ be sequences defined by
\begin{align*}
\left.\begin{array}{ll}
x_1=e_1,  \  x_{n-1}=e_{n-1}\\[4pt]
f_1(x)=0,  \   f_{n-1}(x)=\xi_{n-1}, \quad x=\{\xi_n\}\in E
\end{array}\!\!\right\}, \ n=2,3,4.....
\end{align*}
Then, $(x_n,f_n)$ is a Schauder frame for $E$. But
$\inf\limits_{1\le n<\infty}\|f_n\|=0$. Thus,  $(x_n,f_n)$ is not a
$\Phi$-Schauder frame for $E$.
\end{enumerate}
\end{example}

\begin{remark}
A  $\Phi$-Schauder frame $(x_n,f_n)$   for a Banach space $E$ also
depends on $E$. Indeed,  the pair $(x_n,f_n)$ in Example 2.1(a) is a
$\Phi$-Schauder frame for $l^1$ but not for $c_0$.

\end{remark}

The following theorem gives sufficient conditions for a   Schauder
frame to be $\Phi$-Schauder frame.

\begin{theorem} Let $(x_n,f_n)$ $(\{x_n\}\subset E, \{f_n\}\subset E^*)$ be a Schauder frame
for a Banach space $E$ with  $\inf\limits_{1\le n<\infty}\|f_n\|>0$
and $z_0$  a given non- zero vector in $E$. If there exists no
reconstruction operator $W$ such that $(\{x_n+z_0\},W)$ is a retro
Banach frame for $E^*$ , then  $(x_n,f_n)$ is $\Phi$-Schauder frame.
\end{theorem}
\proof If there exists no reconstruction operator $W$ such that
$(\{x_n+z_0\},W)$  is  a retro Banach frame for $E^*$, then by Lemma
$1.1$ there exists a non-zero functional $\phi$ in $E^*$ such that
$\phi(x_n+z_0)=0$, for all $n\in \mathbb{N}$. Since $(x_n,f_n)$ is
Schauder frame for $E$, so $\phi(x_k)\ne0$ for some $k$. Thus,
$\phi(z_0)\ne0$. Put $\Phi=-\frac{\phi}{\phi(z_0)}$. Then, $\Phi$ is
a functional in $E^*$ such that $\Phi(x_n)=1$, for all $n\in
\mathbb{N}$. Hence $(x_n,f_n)$ is a $\Phi$-Schauder frame .

\endproof

\begin{remark}
The conditions in  Theorem $2.1$  are  not necessary. Indeed, let
$(x_n,f_n)$ be the $\Phi$-Schauder frame for $E=l^1$ given in
Example $2.1(a)$ and $z_0=-e_1$. Then,  by Lemma $1.1$ there exists
a reconstruction operator $W:(E^*)_d=\{\{f(x_n+z_0)\}:f\in E^*\}\to
E^*$  such that $(\{x_n+z_0\},W)$ is a retro Banach frame for $E^*$
with bounds $A=B=1$.
\end{remark}

\begin{remark} The conditions in  Theorem $2.1$ turns out to be
 necessary provided  \break $\Phi(z_0)=-1$. Under this condition  there exists no
 reconstruction operator $\Theta_0$ such that $(\{x_n+z_0\},  \Theta_0)$
 is a retro Banach frame for $E^*$. Since otherwise by  retro frame inequality for
  $(\{x_n+z_0\},\Theta_0)$ and using  $\Phi(x_n+z_0)=0$, for all $n\in \mathbb{N}$,
  we obtain $\Phi=0$, a \break contradiction.\\
 Thus, we have  following theorem.
\end{remark}

\begin{theorem}
Let $(x_n,f_n)$ be  a $\Phi$-Schauder frame  for a Banach space $E$
and that  $z_0$ be a non zero vector in $E$ such that
$\Phi(z_0)=-1$. Then, there exists no reconstruction operator
$\Theta_0$ such that $(\{x_n+z_0\},  \Theta_0)$ is a retro Banach
frame for $E^*$.

\end{theorem}

The following theorem shows that if $E$ and $F$ are Banach spaces
having \break $\Phi$-Schauder frames , then their product space
$E\times F$  with a suitable norm also has a $\Phi$-Schauder frame .

\begin{theorem}
Let $(x_n,f_n)$ $(\{x_n\}\subset E, \{f_n\}\subset E^*)$ and
$(y_n,g_n)$ $(\{y_n\}\subset F,\{g_n\}\subset F^*)$ be Schauder
frames for Banach spaces $E$ and $F$, respectively. Then there exist
\break  sequences $\{z_n\}\subset E\times F$  and $\{h_n\}\subset
(E\times F)^*$ such that $(z_n,h_n)$ is a Schauder frame for \break
$E\times F$. Furthermore, if $(x_n,f_n)$ and $(y_n,g_n)$ are
$\Phi$-Schauder frames, then $(z_n,h_n)$ is also a $\Phi$-Schauder
frame.
\end{theorem}

\begin{proof}
Let $\{z_n\}\subset E\times F$ and $\{h_n\}\subset (E\times F)^*$ be
sequences defined by
\begin{align*}
&\begin{cases}
z_{2n}=(0,y_n),\\
z_{2n-1}=(x_n,0)
\end{cases}\quad n\in \mathbb{N},\\
&\begin{cases}
h_{2n}(x,y)=g_n(y),\\
h_{2n-1}(x,y)=f_n(x)
\end{cases}\quad n\in \mathbb{N.
}
\end{align*}
Then
\begin{align*}
(x,y)
&=\bigg(\sum_{n=1}^\infty f_n(x)x_n,\sum_{n=1}^\infty g_n(y)y_n\bigg)\\
&=\sum_{n=1}^\infty h_n(x,y)z_n,\quad\text{for all } \ (x,y)\in
E\times F.
\end{align*}
Hence $(z_n,h_n)$ is a Schauder frame for $E\times F$.

Further, suppose $(x_n,f_n)$ and $(y_n,g_n)$ are $\Phi$-Schauder
fames. Then, by nature of \break construction of $(z_n,h_n)$, there
exists a functional $\Phi_0 \in  (E\times F)^*$  such that
$\Phi_0(z_n)=1$, for all $n\in \mathbb{N}$ and
 $\inf\limits_{1\le n<\infty}\|h_n\|>0$. Therefore, $(z_n,h_n)$ is $\Phi_0$-Schauder frame.
\end{proof}

\section{ Concluding remarks }
Let $(x_n,f_n)$ be  a Schauder frame for $E$. Then,  there exist a
reconstruction   operator  $\Theta_0$ such that $(\{x_n\},\Theta_0)$
is a retro Banach frame for $E^*$. We say that  $(\{x_n\},
\Theta_0)$ is an  \emph{associated retro Banach frame} of the
Schauder frame $(x_n,f_n)$.\\
 Recently, retro Banach frames of type
$P$ for Banach spaces introduced and  studied in [15]: A retro
Banach frame $(\{x_n\},T)$ for  $E^*$ is said to be of \emph{type
$P$} if it is exact   and there  exists a functional $\Psi\in E^*$  such that $\Psi(x_n)=1$,
for all $n\in \mathbb{N}$.\\

Conclusion:  For a  given  $\Phi$-Schauder frame of $E$,   its
associated retro Banach frame need not be of type $P$  and
vice-versa. This is given in the form of remarks.

\begin{remark}
Let $(x_n,f_n)$ be  a $\Phi$-Schauder frame  for $E$. Then,
associated retro Banach frame   $(\{x_n\},\Theta_0)$, in general,
not of type $P$. Indeed, let  $(x_n,f_n)$  $(\{x_n\}\subset E,
\{f_n\}\subset E^*)$ be  a system for   $E=l^1$ given in Example
$2.1(a)$. Then, $(x_n,f_n)$  is   a $\Phi$-Schauder frame  for $E$
but the associated retro Banach frame  $(\{x_n\},\Theta_0)$ is not
of type $P$.
\end{remark}

\begin{remark} Let  $(\{x_n\}, \Theta)$  be  a retro Banach frame of type $P$ for $E$.
Then, in ganeral, $E$ has no Schauder frame. Furthermore,  if $E$
 admit an associated  Schauder frame , say  $(x_n,f_n)$. Then,
 in general, $(x_n,f_n)$ is not a $\Phi$-Schauder frame  for $E$.
 Indeed, let $E=l^1$ and  $\{x_n\}\subset E$, $\{f_n\}\subset E^*$ be sequences defined by

\item
\begin{align*}
\left.\begin{array}{ll}
x_n=ne_n,\\[4pt]
f_n(x)=\frac{1}{n}\xi_n,\quad x=\{\xi_n\}\in E
\end{array}\!\!\right\} ,\ n=1,2,3,......
\end{align*}

Then, $(x_n,f_n)$ is a Schauder frame for $E$.  Also , there exists
a reconstruction operator $\Theta_0$ such that $(\{x_n\},\Theta_0)$
is a retro Banach frame [associated retro Banach frame of the
Schauder frame  $(x_n,f_n)$]  for $E^*$ which is of type $P$. But
$(x_n,f_n)$ is not  a $\Phi$-Schauder frame  for $E$.
\end{remark}

\textbf{Acknowledgement:} I wish to thank Prof. S. K. Kaushik for
fruitful  discussions.

\mbox{}
\end{document}